\documentclass[a4paper,11pt,reqno]{amsart}
\usepackage{geometry} 
\usepackage{mathtools}
\usepackage{graphicx}
\usepackage{paralist}
\usepackage[utf8]{inputenc}
\usepackage[T1]{fontenc}
\usepackage{csquotes}
\usepackage[all]{xy}
\usepackage{chngcntr}
\usepackage{amsthm}
\usepackage{amsmath}
\usepackage{amsfonts}
\usepackage{amssymb}
\usepackage{mathrsfs}
\usepackage{MnSymbol}
\usepackage{tikz-cd}
\usepackage{tikz, tikz-cd}
\usepackage{mathdots}
\usepackage{stmaryrd}
\usepackage{hyperref}
\usepackage{cleveref}
\usepackage{fancyhdr}
\usepackage{xcolor}

\newtheorem{thm}{Theorem}[section]
\newtheorem{lem}[thm]{Lemma}
\newtheorem{prop}[thm]{Proposition}
\newtheorem{cor}[thm]{Corollary}
\newtheorem{rem}[thm]{Remark}
\theoremstyle{remark}

\theoremstyle{claim}

\newtheorem*{thm*}{Theorem}

\theoremstyle{definition}
\newtheorem{definition}[thm]{Definition}

\theoremstyle{example}

\theoremstyle{convention}

\counterwithin{equation}{section}

\theoremstyle{convention}

\newcommand{\bbN}{\mathbb{N}}
\newcommand{\bbZ}{\mathbb{Z}}

\newcommand{\bbR}{\mathbb{R}}
\newcommand{\bbC}{\mathbb{C}}
\newcommand{\bbH}{\mathbb{H}}

\newcommand{\vol}{\mathrm{vol}}
\newcommand{\Ric}{\mathrm{Ric}}
\newcommand{\inj}{\mathrm{inj}}

\DeclareMathOperator{\Sym}{Sym}
\DeclareMathOperator{\tr}{tr}

\calclayout

\begin{document}
\title[Effective stability of negatively curved Einstein metrics]
{Effective stability of negatively curved Einstein metrics in dimensions at least $4$}
\author{Frieder J\"ackel}
\thanks
{AMS subject classification: 53C20, 53C25 \\
The author was supported by the DFG priority program "Geometry at infinity"}
\date{December 29, 2025}

\begin{abstract}
We show that if a closed manifold of dimension at least four admits a negatively curved metric that is almost Einstein in a suitable sense, then it admits a genuine Einstein metric of negative sectional curvature. Importantly, the pinching constant measuring the almost-Einstein condition neither depends on an upper bound for the diameter or volume, nor on a lower bound for the injectivity radius.
\end{abstract}

\maketitle

\tableofcontents

\section{Introduction}\label{Sec: Introduction}

\subsection{Statement of the main results} 

A classic method for the construction of Einstein metrics is to start with a metric $\bar{g}$ that is almost Einstein in a suitable sense, and then obtain an Einstein metric from $\bar{g}$ by a suitable perturbation procedure. This can for example be done using the Ricci flow (see for example \cite{MO90}). Another possibility is to apply the inverse function theorem to the so-called \textit{Einstein operator} (see for example \cite{Biq00},\cite{And06},\cite{Bam12},\cite{FP20},\cite{HJ24}). Recently, closely following an unpublished preprint of Tian \cite{Tian}, Hamenst\"adt and the author \cite{HJ22} used the Einstein operator to show that metrics with sectional curvature close to $-1$ and whose injectivity radius is uniformly bounded from below can be perturbed to an Einstein metric. The main goal of this note is to extend \cite[Theorem 1]{HJ22} to closed manifolds with arbitrary small injectivity radius.

\begin{thm}\label{Main Theorem - introduction}
For all \(n \geq 4\), \(\alpha \in (0,1)\), \(K > 0\), and \(\Lambda \geq 0\) there exist positive constants $\varepsilon_0=\varepsilon_0(n,\alpha,K,\Lambda)>0$ and $C=C(n,\alpha,K,\Lambda) >0$ with the following property. Let \(M\) be a closed \(n\)-manifold that admits a complete Riemannian metric \(\bar{g}\) satisfying the following conditions for some \(\varepsilon \leq \varepsilon_0\):
\begin{enumerate}[(i)]
\item $||{\rm Ric}(\bar{g})+(n-1)g||_{C^0(M,\bar{g})} \leq \varepsilon$ and $||{\rm Ric}(\bar{g})+(n-1)g||_{L^2(M,\bar{g})} \leq \varepsilon$;
\item $\mathrm{sec}(M,\bar{g}) \leq -K < 0$;
\item \(\sec(M,\bar{g})=-1\) in \(M_{\rm thin}\);
\item \(|| {\nabla} {\rm Ric}(\bar{g})||_{C^0(M,\bar{g})} \leq \Lambda\).
\end{enumerate}
Then there exists an Einstein metric \(g_0\) on \(M\) with \({\rm Ric}(g_0)=-(n-1)g_0\) so that
\[
	||g_0-\bar{g}||_{C^{2,\alpha}(M,\bar{g})} \leq C \varepsilon^{1-\alpha}.
\]
\end{thm}

The key point is that the constants $\varepsilon_0$ and $C$ neither depend on an upper bound for the diameter or volume, nor on a lower bound for the injectivity radius.

The analogue of \Cref{Main Theorem - introduction} for $n=3$ follows from \cite[Theorem 2]{HJ22}. It follows from \cite[Proposition 6.1]{HJ22} that the analogue of \Cref{Main Theorem - introduction} for $n=3$ is wrong if assumtpion (iii) is weakened to a pinching condition $\sec(M,\bar{g}) \in (-1-\varepsilon,-1+\varepsilon)$. However, the construction of these counterexamples exploits the fact that $n=3$. So it might be possible to weaken the assumption (iii) in \Cref{Main Theorem - introduction} for $n \geq 4$.

We have the following immediate consequence of \Cref{Main Theorem - introduction}, which states that if in dimension \(n \geq 4\) a negatively curved almost-Einstein metric is already hyperbolic except in a region of bounded geometry, then it is close to an Einstein metric.

\begin{cor}\label{cor: bounded negative curvature}
For all \(n \geq 4\), \(\alpha \in (0,1)\), \(K>0\), \(\Lambda \geq 0\), \(\iota > 0\) and \(v >0\) there exist \(\varepsilon_1=\varepsilon_1(n,\alpha,K,\Lambda, \iota, v)>0\) and \(C=C(n,\alpha,K,\Lambda,\iota,v)\) with the following property. Let \(M\) be a closed \(n\)-manifold, and let \(\bar{g}\) be a Riemannian metric on \(M\) satisfying
\[
	||{\rm Ric}(\bar{g})+(n-1)\bar{g}||_{C^0(M,\bar{g})} \leq \varepsilon,
	\quad
	\sec(M,\bar{g}) \leq -K
	\quad \text{and} \quad
	 || \nabla {\rm Ric}(\bar{g})||_{C^0(M)} \leq \Lambda
\]
for some \(\varepsilon \leq \varepsilon_1\). Assume that there is \(\Omega \subseteq M\) so that
\[
	{\rm inj}(\Omega) \geq \iota, \quad {\rm vol}(\Omega) \leq v \quad \text{and} \quad \sec(M,\bar{g})=-1 \, \text{ outside } \, \Omega.
\]
Then there exists an Einstein metric \(g_0\) on \(M\) with \({\rm Ric}(g_0)=-(n-1)g_0\) satisyfing
\[
	||g_0-\bar{g}||_{C^{2,\alpha}(M,\bar{g})} \leq C \varepsilon^{1-\alpha}.
\]
\end{cor}

In particular, in dimension \(n \geq 4\), if a $(1+\varepsilon)$-pinched negatively curved metric is already hyperbolic except in a region of bounded geometry, then it is close to a negatively curved Einstein metric. The analogous result for $n=3$ follows from \cite[Theorem 2]{HJ22}.

\begin{cor}\label{cor: almost hyperbolic}
For all \(n \geq 4\), \(\alpha \in (0,1)\), \(\Lambda \geq 0\), \(\iota > 0\) and \(v >0\) there exist positive constants \(\varepsilon_1=\varepsilon_1(n,\alpha,\Lambda, \iota, v)>0\) and \(C=C(n,\alpha,\Lambda,\iota,v) > 0\) with the following property. Let \(M\) be a closed \(n\)-manifold, and let \(\bar{g}\) be a Riemannian metric on \(M\) satisfying
\[
	|\sec(M,\bar{g})+1| \leq \varepsilon \quad \text{and} \quad || \nabla {\rm Ric}(\bar{g})||_{C^0(M)} \leq \Lambda
\]
for some \(\varepsilon \leq \varepsilon_1\). Assume that there is \(\Omega \subseteq M\) such that
\[
	{\rm inj}(\Omega) \geq \iota, \quad {\rm vol}(\Omega) \leq v \quad \text{and} \quad \sec(M,\bar{g})=-1 \, \text{ outside } \, \Omega.
\]
Then there exists an Einstein metric \(g_0\) on \(M\) with \({\rm Ric}(g_0)=-(n-1)g_0\) satisyfing
\[
	||g_0-\bar{g}||_{C^{2,\alpha}(M,\bar{g})} \leq C \varepsilon^{1-\alpha}.
\]
\end{cor}

Previously known results in this direction have to assume that the entire manifold has bounded geometry, though they do \textit{not} assume a bound on $\nabla \Ric$ (see for example \cite[Proposition 3.4]{And90}, \cite[Corollary 1.6]{PW97}, \cite[Theorems 11.4.16 and 11.4.17]{Pet16}). These results are proved by an argument by contradiction, building on suitable convergence theories. In contrast, \Cref{cor: bounded negative curvature} is proved directly. Thus, compared to previously known results, \Cref{cor: bounded negative curvature} (and its proof) is better in some aspects but worse in others.

\subsection{Strategy}
%The basic reason why one might hope that a strategy involving the inverse function theorem could be fruitful in this setting of \Cref{Main Theorem - introduction} is a classic result of Koiso \cite[Theorem 3.3]{Koi78} stating that on a closed manifold of dimension $n \geq 3$, Einstein metrics with negative sectional curvature are isolated in the moduli space of all Riemannian structures (also see \cite[Corollary 12.73]{Bes08}). To prove \Cref{Main Theorem - introduction} we combine this classic result of Koiso (or rather elements of its proof), a celebrated result of Sullivan \cite{Sul87}, and a new geometric preimage counting result. 

Given a background metric $\bar{g}$, we consider the Einstein operator $\Phi_{\bar{g}}$. The zeros of $\Phi_{\bar{g}}$ are Einstein metrics, and when $\bar{g}$ is almost-Einstein (in a suitable sense), then $\Phi_{\bar{g}}(\bar{g})$ is almost zero (in a suitable Banach space). So one might hope to find Einstein metrics close to $\bar{g}$ by an application of the inverse function theorem.

If $\bar{g}$ is a negatively curved almost-Einstein metric, then the linearized Einstein operator $\mathcal{L} \coloneq (D\Phi_{\bar{g}})_{\bar{g}}$ has an $L^2$-spectral gap by the classic work of Koiso \cite{Koi78}. Moreover, using Moser iteration one can obtain $C^0$-estimates for $\mathcal{L}$ from $L^2$-estimates. However, a priori one can apply Moser iteration only in regions of the manifold whose injectivity radius is uniformly bounded from below.

To deal with the regions of small injectivity radius, we apply Moser iteration to the lifted operator in the universal cover $\tilde{M}$ of $M$. The local $L^2$-norms in $\tilde{M}$ and $M$ are related by
\[
	\int_{B(\tilde{x},1)}|\tilde{h}|^2(\tilde{y}) \, d{\rm vol}(\tilde{y})
	\leq \#\left(B(\tilde{x},1) \cap \pi^{-1}(x)\right)
	\int_{B(x,1)}|h|^2(y) \, d{\rm vol}(y),
\]
where $\pi \colon \tilde{M} \to M$ is the universal covering projection. The key ingredient is a new local preimage counting result that bounds the number of local preimages $\#\left(B(\tilde{x},1) \cap \pi^{-1}(x)\right)$ in terms of $d(x,M_{\rm thick})$ (see \Cref{Preimage counting - prop}). 

We then also exploit the fact that $L^2$-spectral gaps imply weighted $L^2$-estimates with exponentially decaying weights whose decay rate depends on the spectral gap. However, in dimensions $n \geq 13$, Koiso's $L^2$-spectral gap is too weak in order for these exponentially decaying weights to absorb the factor $\#\left(B(\tilde{x},1) \cap \pi^{-1}(x)\right)$. We circumvent this problem by improving Koiso's estimate for infinite hyperbolic tubes. Indeed, we achieve this by a simple application of a celebrated result of Sullivan \cite{Sul87} about the first eigenvalue of the Laplacian (acting on functions) on hyperbolic manifolds.

\subsection{Structure of the article}

This article is organized as follows. In \Cref{sec: Preliminaries} we review the necessary preliminaries. Namely, in \Cref{subsec: Einstein operator} we introduce the Einstein operator, while \Cref{subsec: C^0 estimate} and \Cref{subsec: L2-estimate} contain the basic $C^0$- and $L^2$-estimates for its linearization, and \Cref{subsec: cut-off} presents the construction of well-behaved cut-off functions. In \Cref{sec: Preimage counting} we prove the geometric local preimage counting result, which is the key ingredient in the proof of \Cref{Main Theorem - introduction}. This is then used in \Cref{sec: Invertibility of L} to show that the linearized Einstein operator is invertible with respect to suitable Banach norms. Finally, the proofs of the main results are presented in \Cref{sec: Proof of main results}.

\bigskip

\noindent
\textbf{Acknowledgements:} I thank Hans-Joachim Hein for helpful conversations, John M.~Lee for a helpful email exchange, and Ursula Hamenst\"adt and Tristan Ozuch-Meersseman for useful comments regarding an earlier version of this article.

\section{Preliminaries}\label{sec: Preliminaries}

\subsection{The Einstein operator}\label{subsec: Einstein operator}
As mentioned in the introduction, we shall construct the Einstein metric by an application of the inverse function theorem for the so-called \emph{Einstein operator} (see \cite[Section I.1.C]{Biq00}, \cite[page 228]{And06} for more information).
This operator is defined as follows. 

Consider the operator $g \mapsto \Ric(g)+(n-1)g$ acting on smooth Riemannian metrics $g$ on a manifold $M$. This operator is ${\rm Diff}(M)$-equivariant, and thus its linearization is \textit{not} elliptic. To resolve this problem, for a given background metric \(\bar{g}\) one defines the \textit{Einstein operator} $\Phi_{\bar{g}}$ (in Bianchi gauge relative to $\bar{g}$) by 
\begin{equation*}\label{eq: Def of Phi}
	\Phi_{\bar{g}}(g):=\Ric(g)+(n-1)g+\frac{1}{2}\mathcal{L}_{(\beta_{\bar{g}}(g))^\sharp}(g),
\end{equation*}
where the musical isomorphism 
\(\sharp\) is with respect to the metric \(g\), and $\beta_{\bar{g}}$ is the \textit{Bianchi operator} of $\bar{g}$ acting on \((0,2)\)-tensors \(h\) by
\begin{equation*}\label{bianchi}
	\beta_{\bar{g}}(h):=\delta_{\bar{g}}(h)+\frac{1}{2}d{\rm tr}_{\bar{g}}(h):=-\sum_{i=1}^n(\nabla_{e_i}h)(\cdot, e_i)+\frac{1}{2}d{\rm tr}_{\bar{g}}(h).
\end{equation*}
Invoking the formula for the linearization of \(\Ric\) (\cite[Proposition 2.3.7]{Top06}) shows that the linearization of \(\Phi_{\bar{g}}\) at \(\bar{g}\) is given by
\begin{equation}\label{eq: linearisation of Einstein}
	\mathcal{L}h \coloneq (D\Phi_{\bar{g}})_{\bar{g}}(h)=\frac{1}{2}\Delta_L h +(n-1)h.
\end{equation}
Here $\Delta_L$ is the \textit{Lichnerowicz Laplacian} acting on $(0,2)$-tensors $h$ by
\[
    \Delta_Lh=\nabla^\ast \nabla h + \Ric(h),
\]
where $\nabla^\ast \nabla$ is the Connection Laplacian and $\Ric$ is the \textit{Weitzenböck curvature operator} given by
\(
	\Ric(h)(x,y)=h(\Ric(x),y)+h(x,\Ric(y))-2 \,{\rm tr}\, h(\cdot, R(\cdot,x)y)
\)
(see \cite[Section 9.3.2]{Pet16}). \Cref{eq: linearisation of Einstein} shows that \((D\Phi_{\bar{g}})_{\bar{g}}\) is an elliptic operator. This opens up the possibility for an application of the inverse function theorem.

The main point is that the Einstein operator can detect Einstein metrics. The following result can for example be found in \cite[Lemma 2.1]{And06}.

\begin{lem}[Detecting Einstein metrics]\label{Zeros of Phi are Einstein}
Let \((M,\bar{g})\) be a complete Riemannian manifold, and let \(g\) be another 	metric on \(M\) so that 
\[
	\sup_{x \in M}|\beta_{\bar{g}}(g)|(x)<\infty  \quad \text{and} \quad \Ric(g) \leq \lambda g   \,\text{ for some } \lambda < 0,
\]
 where \(\beta_{\bar{g}}(\cdot)\) is the Bianchi operator of the background metric \(\bar{g}\). Denote by \(\Phi_{\bar{g}}\) the Einstein operator defined in (\ref{eq: Def of Phi}). Then
\begin{equation*}
 \Phi_{\bar{g}}(g)=0 \quad \text{if and only if } \quad g  \text{ solves the system} \quad \begin{cases} \Ric(g)=-(n-1)g \\
	\beta_{\bar{g}}(g)=0
	\end{cases}.
\end{equation*}
\end{lem}

\subsection{$C^0$-estimate}\label{subsec: C^0 estimate}

To obtain $C^0$-estimates for the linearization of the Einstein operator, we use the  
De Giorgi--Nash--Moser estimates in the following form. In its formulation, $\Sym^2(T^*M)$ denotes the bundle of symmetric $(0,2)$-tensors on $M$.

\begin{lem}[$C^0$-estimate]\label{lem: Nash-Moser} 
For all $n \in \bbN$, \(\alpha \in (0,1)\), \(\Lambda \geq 0\), and \(\iota>0\) there exist constants $\rho=\rho(n,\alpha,\Lambda,\iota) > 0$ and $C=C(n,\alpha,\Lambda,\iota)>0$ with the following property.  Let \(M\) be a Riemannian \(n\)-manifold satisfying 
\[
		|\sec(M)| \leq \Lambda  \quad \text{and} \quad \inj(M) \geq \iota.
\] 
Let \(f \in C^0\big(\Sym^2(T^*M)\big)\) be arbitrary, and assume \(h \in C^2\big({\rm  Sym}^2(T^*M)\big)\) is a solution of
\[
	\frac{1}{2}\Delta_L h+(n-1)h=f.
\]
Then it holds
\begin{equation*}\label{eq: Nash-Moser} 
	|h|(x) \leq C\Big( ||h||_{L^2(B(x,\rho))}+||f||_{C^0(B(x,\rho))}\Big)
\end{equation*}
for all \(x \in M\).
\end{lem}

The main ingredients for the proof are the classic De Giorgi--Nash--Moser estimates (see for example \cite[Theorem 8.17]{GT01}) and a result by Jost--Karcher \cite[Satz 5.1]{JK82} or Anderson \cite[Main Lemma 2.2]{And90} stating that, under the geometric assumptions, around every point there exists a harmonic chart of a priori size with good analytic control. We refer the reader to \cite[Proof of Lemma 2.2]{HJ24} for further details.

For the proof of \Cref{Main Theorem - introduction} we can \textit{not} directly apply \Cref{lem: Nash-Moser} because the latter assumes a positive lower bound for the injectivity radius. To remedy this problem, we will apply \Cref{lem: Nash-Moser} to the lifted equation \(\mathcal{L}\tilde{h}=\tilde{f}\) in the universal cover \(\tilde{M}\). But then one needs to relate the local \(L^2\)-norm \(||\tilde{h}||_{L^2(B(\tilde{x},1/2))}\) in the universal cover to the local \(L^2\)-norm \(||h||_{L^2(B(x,1/2))}\) in the manifold \(M\) itself. The following basic observation states that this is possible if one can count the number of local preimages. In its formulation, \(\pi:\tilde{M} \to M\) denotes the universal covering projection.

\begin{lem}\label{lem: local L^2-norms and preimages} Let \(x\in M\) and assume \(\omega:M \to \bbR\) is a continuous function satisfying
\[
	\# \left(\pi^{-1}(y) \cap B(\tilde{y},1) \right) \leq \omega(y)
\]
for all \(y \in B(x,1/2) \subseteq M\) and every lift \(\tilde{y} \in \tilde{M}\) of \(y\). Let \(u: M \to \bbR_{\geq 0}\) be a non-negative locally-integrable function and denote by \(\tilde{u}:=u \circ \pi\) its lift to the universal cover. Then
\[
	\int_{B(\tilde{x},1/2)} \tilde{u}(\tilde{y}) \, d{\rm vol}_{\tilde{g}}(\tilde{y}) \leq  \int_{B(x,1/2)} \omega(y) u(y) \, d{\rm vol}_{g}(y). 
\]
\end{lem}

\begin{proof}By the triangle inequality, if $\tilde y\in B(\tilde x,1/2)$ then 
$ B(\tilde x,1/2) \subseteq B(\tilde y,1)$. Thus by assumption, 
a point $y\in B(x,1/2)$ has at most $\omega(y)$ preimages in $B(\tilde x,1/2)$. 
Hence the claim holds true for the indicator function \(u=\chi_U\) of a small open 
subset \(U \subseteq B(x,1/2)\). 
By linearity and monotonicity the result follows for all non-negative simple functions. A standard approximation argument completes the proof.
\end{proof}

\subsection{\(L^2\)-estimates}\label{subsec: L2-estimate}

%In this section we review the $L^2$-estimates required for our purposes.

We start by recalling the following classic $L^2$-spectral gap estimate for the linearized Einstein operator (see (\ref{eq: linearisation of Einstein}))
\[
	\mathcal{L}= \frac{1}{2}\Delta_L+(n-1){\rm id}
\] 
of an (almost) Einstein metric with negative sectional curvature due to Koiso \cite[Section 3]{Koi78}. 
Originally, Koiso only formulated the estimate for genuine Einstein metrics. However, the following slightly generalized result easily follows from the proof presented in \cite[Lemma 12.71]{Bes08}.

\begin{lem}[Koiso]\label{lem: Koiso}
Let $(M,\bar{g})$ be a compact Riemannian $n$-manifold satisfying
\[
	||\Ric(\bar{g})+(n-1)\bar{g}||_{C^0(M,\bar{g})} \leq \delta 
	\quad \text{and} \quad
	\sec(M,\bar{g}) \leq - K < 0.
\]
Then, for all $h \in C^2\big(\Sym^2(T^* M)\big)$, we have
\[
	\big((n-2)K-C\delta \big) \int_M |h|^2 \, d\vol \leq 2\int_{M} \langle \mathcal{L}h,h\rangle \, d\vol
\]
for some $C=C(n)>0$.
\end{lem}

It will be important that this estimate can be dramatically improved for infinite hyperbolic tubes. Here, by an infinite hyperbolic tube $T$ we mean the quotient $T=\bbH^n/ \langle \varphi \rangle$ of hyperbolic space by the infinite cyclic group generated by a hyperbolic isometry $\varphi$, i.e., an isometry that acts via non-trivial translation along a geodesic $\gamma \subseteq \bbH^n$.

\begin{lem}[Hyperbolic tube]\label{lem: spectral gap in hyperbolic tube}
Let $T$ be an infinite hyperbolic tube. Then
\[
	\left(\frac{(n-1)^2}{4}-2 \right) \int_T |h|^2 \, d\vol \leq 2\int_T \langle \mathcal{L}h, h \rangle \, d\vol
\]
for all compactly supported symmetric $(0,2)$-tensors $h \in C_c^2\big(\Sym^2(T^* T)\big)$.
\end{lem}

\begin{proof}
This will easily follow from a celebrated result of Sullivan \cite{Sul87} about the spectral gap of the Laplace operator acting on functions on hyperbolic manifolds.

Note that the limit set of an infinite hyperbolic tube $T=\bbH^n/ \langle \varphi \rangle$, i.e., the accumulation points of an orbit of the action $\pi_1(T) \curvearrowright \bbH^n$ at the boundary at infinity $\partial_\infty \bbH^n$, consists of two points (the attracting and repelling fixed points of the hyperbolic isometry $\varphi$). Consequently, the critical exponent $\delta$ of $T$, i.e., the Hausdorff dimension of its limit set, is $\delta=0$. Therefore, it follows from \cite[Theorem 2.17]{Sul87} that the spectral gap of the Laplacian acting on functions on $T$ is $(n-1)^2/4$, i.e., for all $u \in W^{1,2}(T)$ we have
\begin{equation*}
	\frac{(n-1)^2}{4}\int_T u^2 \, d\vol \leq \int_T |\nabla u|^2 \, d\vol.
\end{equation*}
Applying this with $u=|h|$ implies
\begin{equation*}
	\frac{(n-1)^2}{4}\int_T |h|^2 \, d\vol \leq \int_T |\nabla h|^2 \, d\vol.
\end{equation*}
for all $h \in  C_c^2\big(\Sym^2(T^* T)\big)$ (see for example \cite[Lemma 7.7]{Lee06} for details).

Since $\sec(T) \equiv -1$, we have 
\(
	\frac{1}{2}\Ric(h)=-(nh-\tr(h)g_{\rm hyp}),
\)
where ${\rm Ric}$ denotes the Weitzenb\"ock curvature operator (see \cite[Section 9.3.2]{Pet16}). Thus
\[
	2 \langle \mathcal{L}h , h \rangle
	=\langle \nabla^\ast \nabla h+ \Ric(h)+2(n-1)h , h \rangle
	= \langle \nabla^\ast \nabla h, h \rangle -2|h|^2 + 2|{\rm tr}(h)|^2
\]
for all $h \in C^2\big(\Sym^2(T^* T)\big)$. Combined with the above integral estimate, this completes the proof.
\end{proof}

Finally, we recall that a spectral gap implies weighted $L^2$-bounds due to the following identity.

\begin{rem}[Weighted $L^2$-bound]\label{rem: weighted L^2}
Let $M$ be a complete Riemannian $n$-manifold. Then
\[
	2\int_M \langle \mathcal{L}(\varphi h),\varphi h \rangle \, d\vol = 
	2\int_M  \varphi^2\langle \mathcal{L}(h), h \rangle \, d\vol + \int_M |\nabla \varphi|^2 |h|^2 \, d\vol
\]
for all $h \in C_c^2\big(\Sym^2(T^* M)\big)$ and $\varphi \in C^2(M)$.
\end{rem}

Therefore, if $2\mathcal{L}$ has spectral gap $\lambda \geq 0$, then we obtain weighted $L^2$-bounds with weight $\varphi^2=e^{2\rho}$ as long as $\rho \in C^2(M)$ satisfies ${\rm Lip}(\rho) < \sqrt{\lambda}$.

\begin{proof}
Observe
\[
	\nabla^\ast \nabla (\varphi h)=(\nabla^\ast \nabla \varphi)h - 2\tr(d\varphi \otimes \nabla h)+\varphi \nabla^\ast \nabla h.
\]
Consider the vector field $V=\varphi |h|^2 \nabla \varphi$. A computation in a local orthonormal frame shows
\[
	{\rm div}(V)=|\nabla \varphi|^2 |h|^2 + 2\varphi\langle \tr(d\varphi \otimes \nabla h), h\rangle - \varphi (\nabla^\ast \nabla \varphi)|h|^2.
\]
As $\int_M {\rm div}(V) \, d\vol = 0$ by the divergence theorem, the desired identity easily follows.
\end{proof}

\subsection{Cutting off the thin part}\label{subsec: cut-off}

Since we will only consider almost Einstein metrics with negative sectional curvature, their sectional curvatures satisfy $\sec \in [-n,0)$. Thus, once and for all, we fix a Margulis constant $\mu_n$ for all manifolds $M$ with $\sec(M) \in [-n,0)$. %(see \cite[Theorem 9.5]{BGS85})
If such a manifold $M$ is compact, then its \textit{thin part} $M_{\rm thin}=\{x \in M \, | \, \inj(x) \leq \mu_n\}$ is a finite union of so-called \textit{Margulis tubes}, which are tubular neighbourhoods of closed geodesics of length at most $2\mu_n$ (see \cite[Section 10.3]{BGS85} or \cite[page 3]{BCD93} for more information). 

In general, however, the boundary $\partial T$ of a Margulis tube $T$ is \textit{not} smooth. Nonetheless, Buser--Colbois--Dodziuk \cite[Theorem 2.14]{BCD93} showed that $\partial T$ is uniformly close to a smooth hypersurface $H$ whose sectional curvatures are uniformly bounded. We now explain how their arguments yield the existence of a well-behaved cut-off function.

\begin{lem}\label{lem: cut-off fct}
Let $K > 0$, $M$ be a closed Riemannian $n$-manifold with $\sec(M) \in [-n,-K]$, and $T \subseteq M_{\rm thin}$ a Margulis tube. Then there exists a smooth cut-off function $\eta \colon M \to [0,1]$ with the following properties:
\begin{enumerate}[(i)]
	\item $\eta=1$ in $\left\{x \in T \, | \, \inj(x) \leq \frac{\mu_n}{4}\right\}$ and $\eta=0$ in $M \setminus T$;
	\item $||\eta||_{C^2(M)} \leq C(n,K)$.
\end{enumerate}
\end{lem}

Using this cut-off fucntion, for any $h \in C^2\big(\Sym^2(T^* M)\big)$, we can think of $\eta h$ as a compactly supported $(0,2)$-tensor on an infinite tube; thus we will be able to apply \Cref{lem: spectral gap in hyperbolic tube} to $\eta h$ given that $\sec=-1$ in $M_{\rm thin}$.

\begin{proof}
Consider the inner Margulis tube 
\[
	T^\prime \coloneq \left\{x \in T \, | \, \inj(x) \leq \frac{\mu_n}{2}\right\}.
\] 
Applying \cite[Theorem 2.14]{BCD93} to $T^\prime$ yields a smooth hypersurface $H \subseteq T^\prime$ such that, for every $x \in H$, the intersection $\bar{x} \in \partial T^\prime$ of $\partial T^\prime$ with the radial geodesic from the core geodesic through $x$ satisfies $d(x,\bar{x}) \leq \frac{\mu_n}{50}$. Then we have $\inj \geq \frac{24}{50}\mu_n$ on $H$ because $\inj$ is $1$-Lipschitz, and $\inj \leq \frac{\mu_n}{2}$ on $H$ since $H \subseteq T^\prime$.

In fact, by \cite[(2.19) and page 12]{BCD93}, the smooth hypersurface $H$ approximating $\partial T^\prime$ is of the form $H=\{x \in T^\prime \, | \, g(x)=0\}$ for some smooth function $g \colon T \to \bbR$ satisfying
\[
	 \frac{1}{C_1} \leq ||\nabla g|| \leq C_1 
	\quad \text{and} \quad
	||g||_{C^2(M)} \leq C_2
\]
for some constants $C_i=C_i(n,K)$ ($i=1,2$). Consider the vector field $V \coloneq \frac{\nabla g}{|\nabla g|^2}$ on $T$, and denote by $\Phi_t$ its flow. Keeping in mind that $|V| \leq C_1$ and that $\inj$ is $1$-Lipschitz, we note 
\[
	\Phi_t(H) \subseteq \left\{x \in T \, \left| \, \frac{\mu_n}{4} \leq \inj(x) \leq \frac{3\mu_n}{4}\right\}\right.
	\quad \text{for all } |t| \leq \tau \coloneq \frac{\mu_n}{8C_1}.
\] 
Moreover, $g=t$ on $\Phi_t(H)$.

Fix a smooth bump-function $\psi \colon \bbR \to [0,1]$ such that $\psi=1$ on $(-\infty,0]$, $\psi=0$ on $[\tau,\infty)$, and $||\psi||_{C^2(\bbR)} \leq C(\tau)$, where $\tau$ is as in the equation displayed above. After possibly changing the sign of $g$, we may assume without loss of generality that $g < 0$ in the component of $T \setminus H$ containing the core geodesic, and $g > 0$ in the component of $T \setminus H$ containing $\partial T$. Then
\[
	\eta \colon M \to [0,1], \, x \mapsto 
	\begin{cases}
	\psi(g(x)) & \text{if } x \in T \\
	0 & \text{otherwise}
	\end{cases}
\]
has the desired properties.
\end{proof}

\section{Counting preimages}\label{sec: Preimage counting}

In view of \Cref{lem: local L^2-norms and preimages}, in order to apply the De Giorgi--Nash--Moser estimate, we need to bound the number of local preimages in the universal cover. This is contained in the following preimage counting result. 

\begin{prop}\label{Preimage counting - prop}
Let \(M\) be closed Riemannian \(n\)-manifold such that \(\sec(M) \in [-n,0)\). Assume in addition that $\sec=-1$ in $M_{\rm thin}$. Then there is a universal constant \(C=C(n)\) such that for all \(x \in M\) and every lift \(\tilde{x} \in \tilde{M}\) of \(x\) it holds
\begin{equation}\label{eq: Preimage counting}
	\#\left( B(\tilde{x},1) \cap \pi^{-1}(x) \right) \leq C \exp\left(\left\lfloor \frac{n+1}{2} \right\rfloor d(x,M_{\rm thick})\right),
\end{equation}
where \(\pi:\tilde{M} \to M\) is the universal covering projection, and \(B(\tilde{x},1)\) is the ball of radius one with center \(\tilde{x}\) in \(\tilde{M}\).
\end{prop}

The assumption $\sec=-1$ in $M_{\rm thin}$ will allow us to use $\bbH^n$ as a comparison space. Towards this end, we fix a geodesic $\overline{\gamma} \subseteq \bbH^n$ and an isometry $\overline{\varphi} \in {\rm Isom}^{+}(\bbH^n)$ that is a translation along $\overline{\gamma}$ with translation length $\ell > 0$. Moreover, we define ${\rm inj}_{\overline{M}}:\bbH^n \to \bbR$ by
\begin{equation}\label{eq: inj in comparison space}
		{\rm inj}_{\overline{M}}(\bar{x}):=\frac{1}{2}\min_{k \neq 0}d_{\bbH^n}\big(\overline{\varphi}^k(\bar{x}),\bar{x}\big).
\end{equation}
For $R > 0$ we denote by $Z(R)$ the cylinder $\{\bar{y} \in \bbH^n \, | \, d_{\bbH^n}(\bar{y},\overline{\gamma})=R\}$ of radius $R$ around $\overline{\gamma}$. Observe that $Z(R)$ is isometric to $S_{\sinh(R)}^{n-2} \times \bbR$, where $S_{\sinh(R)}^{n-2}$ is the round $(n-2)$-sphere of radius $\sinh(R)$.

The following lemma is the key technical ingredient for the proof of \Cref{Preimage counting - prop}.

\begin{lem}\label{lem: Preimage counting main lemma}
There is a constant \(C=C(n) > 0\) such that for every $\bar{y} \in \bbH^n$ and \(r \geq {\inj}_{\overline{M}}(\bar{y})\) it holds
\begin{equation}\label{Preimage counting - eq2}
	\#\Big\{ k \in \bbZ \, \big| \, d_{Z}\big(\overline{\varphi}^k(\bar{y}),\bar{y}\big) \leq r  \Big\} \leq C \left(\frac{r}{{\rm inj}_{\overline{M}}(\bar{y})} \right)^{\lfloor \frac{n+1}{2} \rfloor},
\end{equation}
where \(d_{Z}\) is the intrinsic distance in the cylinder \(Z \subseteq \bbH^n\) containing \(\bar{y}\).
\end{lem}

The reason for the exponent \(\left\lfloor \frac{n+1}{2} \right\rfloor\) is the following. If \(\varphi:S^{n-2} \times \bbR \to S^{n-2} \times \bbR\) is of the form \(\varphi(v,t)=(Av,t+\tau)\) for some \(A \in {\rm SO}(\bbR^{n-1})\) and \(\tau \in \bbR\), then any orbit of \(\varphi\) is contained in a flat manifold of dimension at most \(\left\lfloor \frac{n+1}{2} \right\rfloor\) because, by elementary linear algebra, any orbit \(\{A^k(v_0)\}_{k \in \bbZ}\) is contained in a flat torus of dimension at most \(\left\lfloor \frac{n+1}{2} \right\rfloor-1\).

\begin{proof}We split the proof into two steps.

\smallskip
\textit{Step 1 (Reduction to a torus):} We denote by \(\overline{\varphi}^{\perp} \in {\rm SO}(\bbR^{n-1})\) the identification of the orthogonal restriction \((d\bar{\varphi})|_{\overline{\gamma}^{\perp}}\) with an isometry of \(\bbR^{n-1}\) via parallel transport. Abbreviate \(R:=d(\bar{y},\overline{\gamma})\), so that \(Z(R)\) is the cylinder containing \(\bar{y}\). Under the isometry \(Z(R) \cong S_{\sinh(R)}^{n-2} \times \bbR\) the restriction \(\bar{\varphi}:Z(R) \to Z(R)\) is given by
\[
	\overline{\varphi}: S_{\sinh(R)}^{n-2} \times \bbR \to S_{\sinh(R)}^{n-2} \times \bbR, (v,t) \mapsto \big(\overline{\varphi}^{\perp}(v),t+\ell \cosh(R)\big).
\]

For ease of notation we only consider the case that $n$ is odd, and we write \(n=2m+1\), so that \(\bbR^{n-1}=\bbR^{2m}\). Since \(\overline{\varphi}^{\perp} \in {\rm SO}(\bbR^{2m})\), we may, after a change of orthonormal basis, assume that \(\overline{\varphi}^{\perp}\) is of the form
\[
	\overline{\varphi}^{\perp}:\bbC^m \to \bbC^m, (z_1,...,z_m) \mapsto \left(e^{i\theta_1}z_1,..., e^{i\theta_m}z_m \right)
\]
for some \(\theta_1,...,\theta_m \in \bbR\). Let \((v_0,t_0)\) be the point corresponding to \(\bar{y}\) under the isometry \(Z(R) \cong S_{\sinh(R)}^{n-2} \times \bbR\), and write \(v_0=\big(z_1^0,...,z_m^0\big) \in S_{\sinh(R)}^{n-2} \subseteq \bbC^m\). Define 
\[
	T_{\bar{y}}^m:=\Big\{\big(\lambda_1 z_1^0,...,\lambda_m z_m^0\big) \, | \, \lambda_i \in S^1 \subseteq \bbC\Big\},
\]
i.e., \(T_{\bar{y}}^m\) is the orbit of \(v_0 \in S_{\sinh(R)}^{n-2}\) under the the isometric action of \(T^m=S^1 \times ... \times S^1\) on \(\bbC^m\). Note that \(T_{\bar{y}}^m\) is isometric to
\[
	T_{\bar{y}}^m \cong S^1_{1}  \times ... \times S^1_{m},
\]
where \(S_j^1:=S^1_{|z_j^0|}\) is the circle of radius \(|z_j^0|\) (here it is understood that \(S^1_{j}=\{{\rm pt}\}\) if \(z_j^0=0\)). Observe that, under the isometry  \(Z(R) \cong S_{\sinh(R)}^{n-2} \times \bbR\), the orbit \(\{\overline{\varphi}^k(\bar{y})\}_{k \in \bbZ}\) is contained in \(T_{\bar{y}}^m \times \bbR\). Moreover, up to universal constants, the intrinsic distance in \(T_{\bar{y}}^m = S^1_1  \times ... \times S^1_{m}\) agrees with the extrinsic distance in \(\bbC^m\) (and hence also with the distance in \(S_{\sinh(R)}^{2m-1}\)). Therefore, it suffices to prove the desired estimate (\ref{Preimage counting - eq2}) with $d_{T_{\bar{y}}^m \times \bbR}$ instead of $d_Z$.

\smallskip
\textit{Step 2 (Volume counting):} 
Let \(0 \leq a \leq m\) be the number of factors \(S_{j}^1\) of \(T_{\bar{y}}^m\) satisfying
\[
	{\rm diam}(S_{j}^1) \leq {\rm inj}_{\overline{M}}(\bar{y}),
\]
where \(\rm diam\) stands for the intrinsic diameter of the circle. We may, after reordering, assume that this is the case for the first \(a\) factors \(S_{1}^1\),..., \(S_{a}^1\) of $T_{\bar{y}}^m \cong S^1_{1}  \times ... \times S^1_{m}$. 

Note \(d_{\bbH^n}(\cdot,\cdot) \leq d_{T_{\bar{y}}^{m} \times \bbR}(\cdot,\cdot)\) since \(T_{\bar{y}}^{m} \times \bbR \subseteq Z(R) \subseteq \bbH^n\). Thus, by the definition (\ref{eq: inj in comparison space}) of ${\rm inj}_{\overline{M}}(\bar{y})$, we have for 
\(
	{\rm inj}_{\overline{M}}(\bar{y}) 
	\leq 
	\frac{1}{2}d_{T_{\bar{y}}^{m} \times \bbR}\big(\bar{y},\overline{\varphi}^k(\bar{y})\big)
\)
for \(k \neq 0\). Hence the balls \(B\big(\overline{\varphi}^k(\bar{y}),{\rm inj}_{\overline{M}}(\bar{y})\big) \subseteq T_{\bar{y}}^{m}\times \bbR\) (\(k \in \bbZ\)) are pairwise disjoint because \(\overline{\varphi}\) is an isometry. Note that, for any \(r \geq 0\), the volume of balls of radius \(r\) in \(T_{\bar{y}}^{m}\times \bbR\) is bounded from above by 
\[
	{\rm vol}_{T_{\bar{y}}^m \times \bbR}\big(B({\rm pt},r) \big) \leq Cr^{m-a+1}\prod_{j=1}^a {\rm diam}(S_j^1).
\]
Since ${\rm inj}_{\overline{M}}(\bar{y}) \geq {\rm diam}(S_j^1)$ for $j=1,\dots,a$ and ${\rm inj}_{\overline{M}}(\bar{y}) \leq {\rm diam}(S_j^1)$ for $j=a+1,\dots,m$, we can also bound the volume of balls of radius ${\rm inj}_{\overline{M}}(\bar{y})$ from below by
\[
	\frac{1}{C}{\rm inj}_{\overline{M}}(\bar{y})^{m-a+1}\prod_{j=1}^a {\rm diam}(S_j^1) 
	\leq {\rm vol}_{T_{\bar{y}}^m \times \bbR} \big(B\big({\rm pt},{\rm inj}_{\overline{M}}(\bar{y}) \big) \big).
\]
Therefore, a volume counting argument shows that for all \(r \geq {\rm inj}_{\overline{M}}(\bar{y})\) we have
\[
	\#\Big\{ k \in \bbZ \, \big| \, d_{T_{\bar{y}}^{m}\times \bbR}\big(\overline{\varphi}^k(\bar{y}),\bar{y} \big) \leq r \Big\} \leq C\left(\frac{r}{{\rm inj}_{\overline{M}}(\bar{y})}\right)^{m-a+1} \leq C\left(\frac{r}{{\rm inj}_M(x)}\right)^{\lfloor \frac{n+1}{2} \rfloor}
\]
where in the second inequality we used \(r \geq {\rm inj}_{\overline{M}}(\bar{y})\) and \(m-a+1 \leq m+1=\lfloor \frac{n+1}{2} \rfloor\) when \(n=2m+1\). Keeping in mind the end of Step 1, this completes the proof.
\end{proof}

We will also need the following elementary result. In its formulation, a \textit{radial geodesic} is a geodesic \(\sigma:[0,\infty) \to \bbH^n\) with \(\sigma(0) \in \overline{\gamma}\) and \(\sigma^\prime(0) \perp \overline{\gamma}\).

\begin{lem}\label{lem: preimage comparison along radial geodesics}
There is a universal constant \(C_0 > 0\) with the following property. Let \(\sigma:[0,\infty) \to \bbH^n\) be a radial geodesic and \(2 \leq R^\prime \leq R < \infty\). Set \(\bar{x}=\sigma(R)\) and \(\bar{x}^\prime=\sigma(R^\prime)\). Then for all \(r \geq 0\) we have
\begin{align*}
	\#\left\{ k \in \bbZ \, \Big| \, d_{Z(R)}\big(\overline{\varphi}^k(\bar{x}),\bar{x}\big) \leq r   \right\}  
	\leq & 
	\#\Big\{ k \in \bbZ \, \big| \, d_{Z(R^\prime)}\big(\overline{\varphi}^k(\bar{x}^\prime),\bar{x}^\prime\big) \leq r  \Big\}  \\
	\leq &
	\#\left\{ k \in \bbZ \, \big| \, d_{Z(R)}\big(\overline{\varphi}^k(\bar{x}),\bar{x}\big) \leq C_0e^{R-R^\prime}r  \right\} ,
\end{align*}
where \(d_{Z(R)}\) and \(d_{Z(R^\prime)}\) are the intrinsic distances in the cylinders \(Z(R)\) and \(Z(R^\prime)\).
\end{lem} 

\begin{proof}This is a straightforward consequence of the fact that the outward radial projection \(Z(R^\prime) \to Z(R)\) is \(\overline{\varphi}\)-equivariant, distance non-decreasing and \(C_0e^{R-R^\prime}\)-Lipschitz. To see the latter two points, observe that under the isometries \(Z(R) \cong S_{\sinh(R)}^{n-2} \times \bbR\) and \(Z(R^\prime) \cong S_{\sinh(R^\prime)}^{n-2} \times \bbR\) the outward radial projection \(Z(R^\prime) \to Z(R)\) is given by 
\[
	(\theta,t) \mapsto \left(\frac{\sinh(R)}{\sinh(R^\prime)}\theta,\frac{\cosh(R)}{\cosh(R^\prime)}t \right).
\]
Thus the outward radial projection is distance non-decreasing. Note that, for \(s \geq 2\), \(\sinh(s)\) and \(\cosh(s)\) agree with \(e^s\) up to universal constant. Since \(R,R^\prime \geq 2\) by assumption, this shows that the outward radial projection is also \(C_0e^{R-R^\prime}\)-Lipschitz. %This completes the proof. 
\end{proof}

We are now in the position to present the proof of \Cref{Preimage counting - prop}.

\begin{proof}[Proof of \Cref{Preimage counting - prop}]
The desired bound (\ref{eq: Preimage counting}) trivially holds when \(x \in M_{\rm thick}\). So it suffices to consider \(x \in M_{\rm thin}\). Let \(T\) be the Margulis tube containing \(x\) with core geodesic \(\gamma\). Fix a lift $\tilde{\gamma} \subseteq \tilde{M}$ of $\gamma$, and let $\tilde{T} \subseteq \tilde{M}$ be the component of $\pi^{-1}(T)$ containing $\tilde{\gamma}$. Let $\varphi \in {\rm Deck}(\pi)$ be the Deck transformation that is a translation along $\tilde{\gamma}$ with translation length $\ell(\gamma)$. Fix a lift $\tilde{x} \in \tilde{T}$ of $x$. Note that \(B(\tilde{x},1) \cap \pi^{-1}(x)=B(\tilde{x},1) \cap \{\varphi^k(\tilde{x})\}_{k \in \bbZ}\).

\smallskip
\textit{Step 1 (Comparison with $\bbH^n$):} Fix a geodesic $\overline{\gamma} \subseteq \bbH^n$. Using exponential normal coordinates around $\tilde{\gamma}\subseteq \tilde{M}$ and $\overline{\gamma} \subseteq \bbH^n$ gives an obvious diffeomorphism $\Phi:\tilde{M} \to \bbH^n$. There exists an isometry $\overline{\varphi} \in {\rm Isom}^{+}(\bbH^n)$ that is a translation along $\overline{\gamma}$ with translation length $\ell(\gamma)$ and such that $\Phi \circ \varphi = \overline{\varphi} \circ \Phi$. We claim that for all $\tilde{y} \in \tilde{T}$ and $k \in \bbZ$ we have
\begin{equation}\label{eq: distance in comparison space}
	d_{\tilde{M}}\big(\tilde{y},\varphi^k(\tilde{y})\big)=d_{\bbH^n}\big(\Phi(\tilde{y}),\overline{\varphi}^k(\Phi(\tilde{y}))\big).
\end{equation}
Since $\sec=-1$ in $M_{\rm thin}$, the restriction $\Phi|_{\tilde{T}}$ is a Riemannian isometry onto its image, and thus it suffices to check that the geodesic segment from $\tilde{y}$ to $\varphi^k(\tilde{y})$ is contained in $\tilde{T}$.

To see this, fix any $\tilde{y} \in \tilde{T}$ and $k_0 \in \bbZ$. Denote by $\sigma:[0,1] \to \tilde{M}$ the geodesic segment from $\tilde{y}$ to $\varphi^{k_0}(\tilde{y})$. Since $\sec(\tilde{M}) \leq 0$, the function $[0,1] \ni t \mapsto d_{\tilde{M}}\big(\sigma(t),\varphi^k(\sigma(t))\big)$ is convex for all $k \in \bbZ$, and thus attains its maximum at $t=0$ or $t=1$. But as $\sigma(1)=\varphi^{k_0}(\sigma(0))$, the values at $t=0$ and $t=1$ coincide. Consequently, for all $t \in [0,1]$ we have
\[
	\frac{1}{2}\min_{k \neq 0}d_{\tilde{M}}\big(\sigma(t),\varphi^k(\sigma(t))\big) \leq \frac{1}{2}\min_{k \neq 0}d_{\tilde{M}}\big(\tilde{y},\varphi^k(\tilde{y})\big)={\rm inj}_M(y) \leq \mu,
\]
where at the end we used that $y:=\pi(\tilde{y}) \in T \subseteq M_{\rm thin}$ as $\tilde{y} \in \tilde{T}$. This implies that ${\rm inj}_M\big(\pi(\sigma(t))\big) \leq \mu$ for all $t \in [0,1]$, and thus $\sigma \subseteq \tilde{T}$, establishing (\ref{eq: distance in comparison space}).

Set $\overline{T}:=\Phi(\tilde{T})$. Because of (\ref{eq: distance in comparison space}), the desired estimate (\ref{eq: Preimage counting}) follows if for all $\bar{x} \in \overline{T}$ we can show
\begin{equation}\label{eq: Preimage counting in comparison space}
	\#\Big\{k \in \bbZ \, | \, d_{\bbH^n}\big(\bar{x},\overline{\varphi}^k(\bar{x})\big) \leq 1 \Big\} \leq C\exp\left(\left\lfloor \frac{n+1}{2} \right\rfloor d_{\bbH^n}\big(\bar{x}, \partial \overline{T}\big)\right).
\end{equation}

Fix $\bar{x} \in \overline{T}$ and choose a minimal geodesic \(\zeta \subseteq \overline{T}\) from \(\bar{x}\) to \(\partial \overline{T}\) (\(\zeta\) will in general \textit{not} be a radial geodesic when \(n \geq 4\)). Denote by \(\bar{y} \in \partial \overline{T}\) the endpoint of \(\zeta\). Abbreviate \(R_x=d_{\bbH^n}(\bar{x},\overline{\gamma})\) and \(R_y=d_{\bbH^n}(\bar{y},\overline{\gamma})\).

\smallskip
\textit{Step 2 (Proving (\ref{eq: Preimage counting in comparison space})):} Towards proving (\ref{eq: Preimage counting in comparison space}) we first establish the following claim.

\smallskip\noindent
\textit{Claim. For every $r \geq 2$ there exists $C_r=C(n,r) > 0$ with the following property. If $\zeta$ passes through the $r$-neighbourhood $N_r(\overline{\gamma})$ of $\overline{\gamma}$, then
\[
	\#\Big\{k \in \bbZ \, | \, d_{\bbH^n}\big(\bar{x},\overline{\varphi}^k(\bar{x})\big) \leq 1 \Big\} \leq C_r\exp\left(\left\lfloor \frac{n+1}{2} \right\rfloor d_{\bbH^n}\big(\bar{x}, \partial \overline{T}\big)\right).
\]}

\smallskip\noindent
\textit{Proof of Claim.}
Clearly, as $\overline{\varphi}$ has translation length $\ell:=\ell(\gamma)$, we see
\begin{equation*}
	\#\Big\{k \in \bbZ \, | \, d_{\bbH^n}\big(\bar{x},\overline{\varphi}^k(\bar{x}) \big) \leq 1\Big\} \leq \frac{C}{\ell}.
\end{equation*}
Since $\bar{y} \in \partial \overline{T}$, i.e., ${\rm inj}_{\overline{M}}(\bar{y})=\mu$, \cite[Lemma 1]{Rez95} states $1/\ell \leq C \exp\left(\left\lfloor \frac{n+1}{2} \right\rfloor R_y \right)$ (this can also be easily deduced from \Cref{lem: Preimage counting main lemma} and \Cref{lem: preimage comparison along radial geodesics}). Finally, $d(\bar{x},\partial \overline{T})=\ell(\zeta) \geq R_y-r$ because $\zeta$ passes through $N_r(\overline{\gamma})$. Combining these inequalities yields the claim.
\hfill$\blacksquare$

\smallskip
Fix $\delta \geq 1$ such that $\bbH^n$ is Gromov $\delta$-hyperbolic. Denote by $\sigma_x$ and $\sigma_y$ the radial geodesics from $\overline{\gamma}$ to $\bar{x}$ and $\bar{y}$. Then, by Gromov $\delta$-hyperbolicity,
\[
	\zeta \subseteq N_{2\delta}(\overline{\gamma}) \cup N_{2\delta}(\sigma_x) \cup N_{2\delta}(\sigma_y).
\] 
If $\zeta$ passes through $N_{12\delta}(\overline{\gamma})$, then the desired estimate (\ref{eq: Preimage counting in comparison space}) follows from the above claim. 

We may thus assume that $\zeta$ and $N_{12\delta}(\overline{\gamma})$ are disjoint, so that $\zeta \subseteq  N_{2\delta}(\sigma_x) \cup N_{2\delta}(\sigma_y)$. Then $\zeta \cap  N_{2\delta}(\sigma_x) \cap N_{2\delta}(\sigma_y) \neq \emptyset$ because $\zeta$ is connected and $N_{2\delta}(\sigma_x), N_{2\delta}(\sigma_y)$ are open. Consequently, there exist $\bar{x}^\prime \in \sigma_x([0,R_x])$ and $\bar{y}^\prime \in \sigma_y([0,R_y])$ such that 
\[
	d_{\bbH^n}(\bar{x}^\prime,\bar{y}^\prime) < 4\delta
	\quad \text{and} \quad 
	d(\bar{y}^\prime,\zeta) < 2\delta. 
\]
For two radial geodesics $\sigma_1,\sigma_2$, the function $ t \mapsto d_{\bbH^n}(\sigma_1(t),\sigma_2(t))$ ($t \geq 0$) is monotone increasing due to the convexity of $d_{\bbH^n}(\cdot,\cdot)$. Thus
\[
	\#\Big\{k \in \bbZ \, | \, d_{\bbH^n}\big(\overline{\varphi}^k(\bar{x}),\bar{x} \big) \leq 1\Big\}
	\leq 
	\#\Big\{k \in \bbZ \, | \, d_{\bbH^n}\big(\overline{\varphi}^k(\bar{x}^\prime),\bar{x}^\prime \big) \leq 1\Big\}.
\]
Using $d_{\bbH^n}(\bar{x}^\prime,\bar{y}^\prime) \leq 4\delta$ and that $\overline{\varphi}$ is an isometry, we see
\[
	\#\Big\{k \in \bbZ \, | \, d_{\bbH^n}\big(\overline{\varphi}^k(\bar{x}^\prime),\bar{x}^\prime \big) \leq 1\Big\}
	\leq 
	\#\Big\{k \in \bbZ \, | \, d_{\bbH^n}\big(\overline{\varphi}^k(\bar{y}^\prime),\bar{y}^\prime \big) \leq 8\delta+1\Big\}.
\]
As $d(\bar{y}^\prime,\overline{\gamma})\geq d(\zeta,\overline{\gamma})-d(\bar{y}^\prime,\zeta) \geq 10\delta \geq (8\delta+1)+1$, there exists $r_0=r_0(\delta) \geq \mu $ such that
\[
	\#\Big\{k \in \bbZ \, | \, d_{\bbH^n}\big(\overline{\varphi}^k(\bar{y}^\prime),\bar{y}^\prime \big) \leq 8\delta+1\Big\}
	\leq 
	\#\Big\{k \in \bbZ \, | \, d_{Z(R^\prime)}\big(\overline{\varphi}^k(\bar{y}^\prime),\bar{y}^\prime \big) \leq r_0\Big\},
\]
where $d_{Z(R^\prime)}$ is the intrinsic distance in the cylinder $Z(R^\prime)$ containing $\bar{y}^\prime$. Appealing first to \Cref{lem: preimage comparison along radial geodesics} and then to \Cref{lem: Preimage counting main lemma} yields
\begin{align*}
	\#\Big\{k \in \bbZ \, | \, d_{Z(R^\prime)}\big(\overline{\varphi}^k(\bar{y}^\prime),\bar{y}^\prime \big) \leq r_0\Big\} \leq &
	\#\Big\{k \in \bbZ \, | \, d_{Z(R_y)}\big(\overline{\varphi}^k(\bar{y}),\bar{y} \big) \leq C_0r_0e^{R_y-R^\prime}\Big\} \\
	\leq & C \left( \frac{C_0r_0 e^{R_y-R^\prime}}{{\rm inj}_{\overline{M}}(\bar{y})}\right)^{\left\lfloor \frac{n+1}{2}\right\rfloor}.
\end{align*}
Note ${\rm inj}_{\overline{M}}(\bar{y})=\mu$ since $\bar{y} \in \partial \overline{T}$, and $R_y-R^\prime=d(\bar{y}^\prime,\bar{y}) \leq 2\delta+\ell(\zeta) = 2\delta+d(\bar{x}, \partial \overline{T})$ by the choice of $\bar{y}^\prime$. Therefore, combining the above inequalities yields (\ref{eq: Preimage counting in comparison space}), and thus completes the proof.
\end{proof}

\section{Invertibility of \(\mathcal{L}\)}\label{sec: Invertibility of L}

In order to apply to inverse function theorem to the Einstein operator $\Phi_{\bar{g}}$, we need to show that its linearization $\mathcal{L}$ at the background metric $\bar{g}$ is invertible. Recall from (\ref{eq: linearisation of Einstein}) that this linearization is
\[
	\mathcal{L}h \coloneq (D\Phi_{\bar{g}})_{\bar{g}}(h)=\frac{1}{2}\Delta_Lh+(n-1)h.
\]
Since the pinching constant $\varepsilon_0$ in \Cref{Main Theorem - introduction} is \emph{not} allowed to depend on geometric quantities such as $\vol(M)$, we have to show that $||\mathcal{L}^{-1}||_{\rm op}$ is bounded from above by a universal constant. To achieve this, we will consider certain \textit{hybrid norms} adapted to our given geometric setting.

\subsection{Hybrid norms}\label{subsec: hybrid norms}
Before we can give the definition of the hybrid norms, we have to fix some constants. For all $n \geq 4$ we define
\begin{equation}\label{eq: def of lambda_0}
	\lambda_0 \coloneq \lambda_0(n) \coloneq \max\left\{n-2,\frac{(n-1)^2}{4}-2\right\}.
\end{equation}
Then, by \Cref{lem: Koiso} and \Cref{lem: spectral gap in hyperbolic tube}, in any $n$-dimensional infinite hyperbolic tube the operator $2\mathcal{L}$ has spectral gap at least $\lambda_0$.

Observe that $2\sqrt{\lambda_0} > \left\lfloor \frac{n+1}{2} \right\rfloor$ for all $n \geq 4$. So we can, once and for all, fix some
\begin{equation}\label{eq: def of beta}
	\beta=\beta(n) \in \big(0,\sqrt{\lambda_0} \big) 
	\quad \text{such that} \quad
	2\beta > \left\lfloor \frac{n+1}{2} \right\rfloor.
\end{equation}
Then, by \Cref{rem: weighted L^2}, in any infinite hyperbolic tube we obtain weighted $L^2$-estimates for $\mathcal{L}$ with weight $\varphi^2$ if $\log(\varphi)$ is $\beta$-Lipschitz.

Let \(M\) be a compact Riemannian manifold of dimension \(n \geq 4\) with $\sec(M) \in [-n,0)$. We abbreviate
\begin{equation}\label{eq: def M'thin}
	M_{\rm thin}^\prime \coloneq \big\{x \in M \, | \, \inj(x) \leq \mu_n/4 \big\}
	\quad \text{and} \quad
	M_{\rm thick}^\prime \coloneq \big\{x \in M \, | \, \inj(x) \geq \mu_n/4 \big\}.
\end{equation}
For any basepoint \(x \in M_{\rm thin}^\prime\) we define the semi-norms
\begin{equation}\label{eq: weighted H^2-norm}
	||h||_{H^2(T_x;\omega_x)}:=\left(\int_{T(x)} e^{-2\beta r_x(y)}\big(|h|^2+|\nabla h|^2+|\Delta h|^2 \big)(y)  \, d{\rm vol}(y)\right)^{\frac{1}{2}}
\end{equation}
and
\begin{equation}\label{eq: weighted L^2-norm}
	||f||_{L^2(T_x;\omega_x)}:=\left(\int_{T(x)} e^{-2\beta r_x(y)}|f|^2(y)  \, d{\rm vol}(y)\right)^{\frac{1}{2}},
\end{equation}
where \(r_x(y)=d_M(x,y)\) and $T(x) \subseteq M_{\rm thin}$ is the Margulis tube containing $x$. %Here the notation \(\omega_x\) should indicate that there is a weight function involved that depends on \(x \in M\).

%The reason why we use the weights \(e^{-(2\sqrt{n-2}-\delta)r_x}\) is that we can only obtain weighted \(L^2\)-estimates with weights \(e^{2\omega}\) for functions \(\omega\) satisfying \(|\nabla \omega| < \sqrt{n-2}\). This is so that, when applying \Cref{lem: weighted L^2-estimate} with $\varphi=e^\omega$, the factor \((n-2)e^{2\omega}|h|^2\) on the left hand side of (\ref{eq: weighted L^2-estimate}) can absord the factor \(|\nabla \omega|^2e^{2\omega}|h|^2\) on the right hand side of (\ref{eq: weighted L^2-estimate}).

We can now give the definition of the hybrid norms. In its formulation, for any Margulis tube $T \subseteq M_{\rm thin}$, we denote by $\eta_T$ the cut-off function given by \Cref{lem: cut-off fct}.

\begin{definition}[Hybrid norms]\label{def: hybridnorm}
For \(\alpha \in (0,1)\) and $k \in \{0,2\}$ the \textit{hybrid norms} \(|| \cdot ||_k \) on \(C^{k,\alpha}\big({\rm Sym}^2(T^*M)\big)\) are defined as
\begin{equation*}
	||h||_2:=\max \left\{
	||h||_{C^{2,\alpha}(M)}
	\, , \, 
	||h||_{H^2(M)}
	\, , \,
	\sup_{x \in M_{\rm thin}^\prime}e^{\frac{1}{2}\lfloor \frac{n+1}{2} \rfloor d(x,M_{\rm thick}^\prime)}\left|\left|\eta_{T(x)}h\right|\right|_{H^2(T_x;\omega_x)} 
	\right\}
\end{equation*}
and
\begin{equation*}
	||f||_0:=\max \left\{
	||f||_{C^{0,\alpha}(M)} 
	\, , \, 
	||f||_{L^2(M)}
	\, , \, 
	\sup_{x \in M_{\rm thin}^\prime}e^{\frac{1}{2}\lfloor \frac{n+1}{2} \rfloor d(x,M_{\rm thick}^\prime)}\left|\left|\eta_{T(x)}f\right|\right|_{L^2(T_x;\omega_x)}\right\}.
\end{equation*}
\end{definition} 

Here we use the following notion of H\"older norm: For a Riemannian manifold $M$ as in \Cref{Main Theorem - introduction}, the universal covering $\tilde{M}$ has infinite injectivity radius and its Ricci tensor has uniformly bounded $C^1$-norm. Thus, by a result of Anderson \cite{And90}, around every point in $\tilde{M}$ there exists a harmonic chart of a priori size and for which the coefficients $\bar{g}_{ij}$ of the metric have controled $C^{2,\alpha}$-norm. The H\"older norm for a tensor on $\tilde{M}$ is defined as the H\"older norm of its coefficients in these harmonic charts. Finally, we can extend this notion to tensors on $M$ by defining their H\"older norm as the H\"older norm of their lift to $\tilde{M}$. For further details we refer the reader to \cite[Proof of Proposition 2.5 and Remark 2.7]{HJ22}.

\subsection{A priori estimate}\label{subsec: a priori for L}

We now prove the main result of this section. Namely, we show that with respect to the hybrid norms $||\cdot||_2$ and $||\cdot||_0$ from \Cref{def: hybridnorm}, the linearized Einstein operator $\mathcal{L}=\frac{1}{2}\Delta_L+(n-1){\rm id}$ is uniformly invertible.

\begin{prop}\label{prop: L is invertible}For all \(n \geq 4\), \(\alpha \in (0,1)\), $K > 0$, and \(\Lambda \geq 0\) there exists a constant \(C=C(n,\alpha,K,\Lambda)\) with the following property. Let \(M\) be a closed Riemannian \(n\)-manifold satisfying
\[
	 \sec(M) \in [-n,-K],
	 \quad 
	 \sec=-1 \text{ in }M_{\rm thin},
	 \quad \text{and} \quad
	  ||\nabla \Ric||_{C^0(M)}\leq \Lambda.
\]
Then the operator
\[
	\mathcal{L}: \Big( C^{2,\alpha}\big( {\rm Sym}^2(T^*M)\big), ||\cdot||_2 \Big) \longrightarrow \Big( C^{0,\alpha}\big( {\rm Sym}^2(T^*M)\big), ||\cdot||_0 \Big)
\]
is invertible, and 
\[
	||\mathcal{L}||_{\rm op}, ||\mathcal{L}^{-1}||_{\rm op} \leq C,
\]
where \(||\cdot||_2\) and \(||\cdot||_0\) are the norms defined in \Cref{def: hybridnorm}.
\end{prop}

\begin{proof}
Throughout the entire proof, we will abbreviate $f:=\mathcal{L}h$.

\smallskip
\textit{Step 1 (Boundedness):}
We start by showing that $\mathcal{L}$ is a bounded linear operator, i.e., $||f||_0 \leq C ||h||_2$ for all $h \in C^{2,\alpha}\big(\Sym^2(T^*M) \big)$. Clearly, $||f||_{C^{0,\alpha}(M)}, ||f||_{L^2(M)} \leq C||h||_2$. So it remains to prove, for all $x \in M_{\rm thin}^\prime$, 
\[
	e^{\frac{1}{2}\lfloor \frac{n+1}{2} \rfloor d(x,M_{\rm thick}^\prime)}||\eta_{T(x)}f||_{L^2(T_x;\omega_x)} \leq C||h||_2.
\]

Let $x \in M_{\rm thin}^\prime$ be arbitrary. We abbreviate $T \coloneq T(x)$ for the Margulis tube containing $x$, and $\eta \coloneq \eta_{T(x)}$ for the cut-off function given by \Cref{lem: cut-off fct}. An easy computation shows 
\begin{equation}\label{eq: L(eta h)}
	\mathcal{L}(\eta h)=\eta f - 2\tr(d\eta \otimes \nabla h)+(\nabla^\ast \nabla \eta)h.
\end{equation}
Clearly, 
\begin{equation}\label{eq: L bounded 1}
	||\mathcal{L}(\eta h)||_{L^2(T_x;\omega_x)} \leq C||\eta h||_{H^2(T_x,\omega_x)}.
\end{equation}
Moreover, as $||\eta||_{C^2(M)} \leq C$ by \Cref{lem: cut-off fct},
\begin{equation}\label{eq: L bounded 2}
	\int_{T}e^{-2\beta r_x} \big(|\tr(d\eta \otimes \nabla h)|^2+|(\nabla^\ast \nabla \eta)h|^2\big) \, d\vol \leq C \int_{{\rm supp}(\nabla \eta)}e^{-2\beta r_x} \big(|\nabla h|^2+|h|^2\big) \, d\vol
\end{equation}
However, ${\rm supp}(\nabla \eta) \subseteq M_{\rm thick}^\prime$ by \Cref{lem: cut-off fct}(i) and the definition (\ref{eq: def M'thin}) of $M^\prime_{\rm thick}$, whence
\begin{equation}\label{eq: L bounded 3}
	\int_{{\rm supp}(\nabla \eta)}e^{-2\beta r_x} \big(|\nabla h|^2+|h|^2\big) \, d\vol \leq e^{-2\beta d(x,M_{\rm thick}^\prime)}\int_M \big(|\nabla h|^2+|h|^2\big) \, d\vol.
\end{equation}
Since $2\beta > \left\lfloor \frac{n+1}{2} \right\rfloor$ by the choice (\ref{eq: def of beta}) of $\beta$, combining (\ref{eq: L(eta h)})-(\ref{eq: L bounded 3}) implies
\begin{align*}
	e^{\frac{1}{2}\left\lfloor \frac{n+1}{2} \right\rfloor d(x,M_{\rm thick}^\prime)}||\eta f||_{L^2(T_x,\omega_x)} 
	\leq &
	C \left(e^{\frac{1}{2}\left\lfloor \frac{n+1}{2} \right\rfloor d(x,M_{\rm thick}^\prime)} ||\eta h||_{L^2(T_x,\omega_x)}+||h||_{H^2(M)} \right) \\
	\leq & C ||h||_2,
\end{align*}
where in the second inequality we used the definition of $||\cdot||_2$ (see \Cref{def: hybridnorm}). This completes the proof that $\mathcal{L}$ is a bounded linear operator. 

\smallskip
\textit{Step 2 (Bounded inverse):}
It will suffice to prove the a priori estimate $||h||_2 \leq C||f||_0$ for all $h \in C^{2,\alpha}\big(\Sym^2(T^*M) \big)$. Indeed, given this a priori estimate, standard arguments show that $\mathcal{L}$ is surjective; consequently $\mathcal{L}$ is invertible and $||\mathcal{L}^{-1}||_{\rm op} \leq C$ thanks to the a priori estimate. 

\smallskip
\textit{Step 2.1 (Integral estimates):}
We know $||h||_{L^2(M)} \leq C||f||_{L^2(M)}$ due to Koiso's spectral gap estimate (\Cref{lem: Koiso}). Standard integration by parts arguments then imply (see \cite[Proof of Proposition 4.3]{HJ22} for more details)
\begin{equation}\label{eq: easy integral estimate}
	||h||_{H^2(M)} \leq C||f||_{L^2(M)}.
\end{equation}

Now let $x \in M_{\rm thin}^\prime$ be arbitrary. We again abbreviate $T \coloneq T(x)$ and $\eta \coloneq \eta_{T(x)}$. Since $\sec=-1$ in $T$, $T$ isometrically embeddeds into an infinite hyperbolic tube $T_\infty$. Therefore, from \Cref{lem: Koiso} and \Cref{lem: spectral gap in hyperbolic tube} we obtain
\begin{equation}\label{eq: L^2 in tube}
	\lambda_0 \int_T |h_T|^2 \, d\vol \leq 2\int_T \langle \mathcal{L}h_T,h_T \rangle \, d\vol
\end{equation}
for all $h_T \in C^{2,\alpha}\big(\Sym^2(T^*M) \big)$ with ${\rm supp}(h_T) \subseteq T$. Here $\lambda_0$ is given by (\ref{eq: def of lambda_0}). 

Thanks to \cite[Theorem 1]{Azagra2007} we may henceforth act as if $r_x=d_M(\cdot,x)$ were smooth. Invoking (\ref{eq: L^2 in tube}) with $h_T=e^{-\beta r_x} \eta h$, and using \Cref{rem: weighted L^2} with $\varphi=e^{-\beta r_x}$ and $\eta h$ instead of $h$, implies
\begin{align*}
	\lambda_0 \int_T e^{-2\beta r_x}|\eta h|^2 \, d\vol \leq & 2\int_T e^{-2\beta r_x}\langle \mathcal{L}(\eta h), \eta h \rangle \, d\vol \\	
	&+ \beta^2\int_T e^{-2\beta r_x} |\eta h|^2 \, d\vol,
\end{align*}
and therefore, as $\beta^2 < \lambda_0$ due to the choice (\ref{eq: def of beta}) of $\beta$,
\begin{equation}\label{eq: weighted L^2 in tube}
	\int_T e^{-2\beta r_x}|\eta h|^2 \, d\vol \leq C_0 \int_T e^{-2\beta r_x}\langle \mathcal{L}(\eta h), \eta h \rangle \, d\vol.
\end{equation}
Recall
\(
	\mathcal{L}(\eta h)=\eta f - 2\tr(d\eta \otimes \nabla h)+(\nabla^\ast \nabla \eta)h.
\)
Using the Cauchy--Schwarz inequality and $2ab \leq a^2+b^2$ for all $a,b \in \bbR$, we can bound $|\langle \eta f, \eta h \rangle| \leq \frac{1}{2C_0}|\eta h|^2 +\frac{C_0}{2}|\eta f|^2$. From (\ref{eq: weighted L^2 in tube}) we thus obtain
\begin{equation}\label{eq: almost weighted L^2}
	\begin{split}
	\int_T e^{-2\beta r_x}|\eta h|^2 \, d\vol \leq & C_0^2 \int_T e^{-2\beta r_x} |\eta f|^2 \, d\vol \\
	&+ 2C_0\int_{T}e^{-2\beta r_x}\langle - 2\tr(d\eta \otimes \nabla h)+(\nabla^\ast \nabla \eta)h, \eta h \rangle \, d\vol.
	\end{split}
\end{equation}
By the exact same arguments as in Step 1, we can bound the second term in (\ref{eq: almost weighted L^2}) by 
\begin{equation}\label{eq: almost weighted L^2 2}
	\int_{T}e^{-2\beta r_x}\langle - 2\tr(d\eta \otimes \nabla h)+(\nabla^\ast \nabla \eta)h, \eta h \rangle \, d\vol \leq e^{-\left\lfloor \frac{n+1}{2} \right\rfloor d(x,M_{\rm thick}^\prime)}||h||_{H^2(M)}^2.
\end{equation}
Combining (\ref{eq: almost weighted L^2}), (\ref{eq: almost weighted L^2 2}), (\ref{eq: easy integral estimate}) we finally obtain
\begin{align}\label{eq: harder integral estimate}
	e^{ \frac{1}{2}\left\lfloor \frac{n+1}{2} \right\rfloor d(x,M_{\rm thick}^\prime)}||\eta h||_{L^2(T_x,\omega_x)} &\leq
	 C \left(e^{ \frac{1}{2}\left\lfloor \frac{n+1}{2} \right\rfloor d(x,M_{\rm thick}^\prime)}||\eta f||_{L^2(T_x,\omega_x)} + ||f||_{L^2(M)} \right) \notag \\
	 &\leq C||f||_0,
\end{align}
where in the second inequality we used the definition of $||\cdot||_0$ (see \Cref{def: hybridnorm}). From the weighted $L^2$-estimate (\ref{eq: harder integral estimate}) one can again deduce weighted $H^2$-estimates 
\begin{equation}\label{eq: harder integral estimate 2}
	e^{ \frac{1}{2}\left\lfloor \frac{n+1}{2} \right\rfloor d(x,M_{\rm thick}^\prime)}||\eta h||_{H^2(T_x,\omega_x)} \leq C||f||_0
\end{equation} 
by standard techniques (see \cite[Proof of Proposition 4.3]{HJ22} for more details). This completes the proof of the integral estimates.

\smallskip
\textit{Step 2.2 (\(C^0\)-estimate):}
It remains to estimate \(||h||_{C^{2,\alpha}(M)}\).
Due to Schauder estimates (see for example \cite[Proposition 2.5]{HJ22}), it suffices to bound \(||h||_{C^0(M)}\) by $||f||_0$. Let $\tilde{f}$ and $\tilde{h}$ be the lifts of $f$ and $h$ to the universal cover $\tilde{M}$. Then $\mathcal{L}\tilde{h}=\tilde{f}$ in $\tilde{M}$. Note that $\tilde{M}$ satisfies the assumptions in \Cref{lem: Nash-Moser} with, say, $\iota = 1$ since $\sec(M) \in [-n,-K]$. We may assume without loss of generality that $\rho=\rho(n,\alpha,\Lambda,\iota) > 0$ given by \Cref{lem: Nash-Moser} is at most $\frac{\mu_n}{8}$. Thus, applying \Cref{lem: Nash-Moser} to $\mathcal{L}\tilde{h}=\tilde{f}$ yields for all $x \in M$ and every lift $\tilde{x} \in \tilde{M}$ of $x$
\begin{align}\label{eq: C^0 from local L^2}
	|h|(x)=|\tilde{h}|(\tilde{x})\leq& C\Big(||\tilde{h}||_{L^2(B(\tilde{x},\mu_n/8))}+||\tilde{f}||_{C^0(\tilde{M})} \Big) \notag\\
	=&C\Big(||\tilde{h}||_{L^2(B(\tilde{x},\mu_n/8))}+||f||_{C^0(M)} \Big)
\end{align}
We now make a case distinction.

We first consider the case $\inj(x) \geq \mu_n/8$. Then 
\[
	||\tilde{h}||_{L^2(B(\tilde{x},\mu_n/8))}=||h||_{L^2(B(x,\mu_n/8))} \leq ||h||_{L^2(M)},
\]
and thus we obtain $|h|(x) \leq C\big(||f||_{L^2(M)}+||f||_{C^0(M)} \big)$ by combining (\ref{eq: easy integral estimate}) and (\ref{eq: C^0 from local L^2}).

Now consider the case $\inj(x) \leq \mu_n/8$. Observe that \Cref{Preimage counting - prop} also holds with $M_{\rm thick}^\prime$ instead of $M_{\rm thick}$ (just replace $\mu_n$ by $\mu_n/4$ throughout \Cref{sec: Preimage counting}). So we can apply \Cref{lem: local L^2-norms and preimages} with 
\[
	\omega(y)=C \exp\left(\left\lfloor \frac{n+1}{2}\right\rfloor d\left(y, M_{\rm thick}^\prime \right)\right). 
\]
Hence
\begin{align}\label{eq: intermediate C^0 from L^2}
	\int_{B(\tilde{x},\mu_n/8)} |\tilde{h}|^2(\tilde{y}) \, d\vol(\tilde{y}) \leq & C\int_{B(x,\mu_n/8)}e^{\left\lfloor \frac{n+1}{2}\right\rfloor d(y, M_{\rm thick}^\prime)}|h|^2(y) \, d\vol(y) \notag \\
	\leq & C e^{\left\lfloor \frac{n+1}{2}\right\rfloor d(x, M_{\rm thick}^\prime)}\int_{B(x,\mu_n/8)}|h|^2(y) \, d\vol(y).
 \end{align}
Since $\inj$ is $1$-Lipschitz, we have $\inj \leq \frac{\mu_n}{4}$ on $B(x,\mu_n/8)$, and hence $\eta_{T(x)}=1$ in $B(x,\mu_n/8)$ by \Cref{lem: cut-off fct}(i). Moreover, $e^{-2\beta r_x}$ is uniformly bounded from below on $B(x,\mu_n/8)$. Thus, by the definition (\ref{eq: weighted L^2-norm}) of the semi-norm $||\cdot||_{L^2(T_x,\omega_x)}$,
\begin{equation*}
	 \int_{B(x,\mu_n/8)}|h|^2 \, d\vol \leq C\int_{T(x)}e^{-2\beta r_x}|\eta_{T(x)} h|^2 \, d\vol = C\left|\left|\eta_{T(x)}h \right|\right|_{L^2(T_x,\omega_x)}^2.
\end{equation*}
Combining this together with (\ref{eq: C^0 from local L^2}), (\ref{eq: intermediate C^0 from L^2}), (\ref{eq: harder integral estimate}) finally yields $|h|(x) \leq C||f||_0$.

Therefore, in either case we have shown $|h|(x) \leq C||f||_0$. As $x \in M$ was arbitrary, this implies $||h||_{C^0(M)} \leq C||f||_0$. Together with Schauder, (\ref{eq: easy integral estimate}), and (\ref{eq: harder integral estimate 2}) this completes the proof of the a priori estimate $||h||_2 \leq C||f||_0$.
\end{proof}

\section{Proof of the main results}\label{sec: Proof of main results}

We can now present the proofs of the results mentioned in the introduction. We start with \Cref{Main Theorem - introduction}, which follows from \Cref{prop: L is invertible} by a straightforward application of the inverse function theorem. 

\begin{proof}[Proof of \Cref{Main Theorem - introduction}]
We equip $C^{k,\alpha}\big(\Sym^2(T^\ast M)\big)$ with the hybrid norm $||\cdot||_k$ defined in \Cref{def: hybridnorm} ($k=0,2$); $B(h,r)$ shall denote the balls with respect to these norms. 
    
Any element in $B(\bar{g},1/2) \subseteq C^{2,\alpha}\big(\Sym^2(T^\ast M)\big)$ is a positive definite $(0,2)$-tensor, that is, a Riemannian metric on $M$. Let $\Phi=\Phi_{\bar{g}}$ be the Einstein operator defined in (\ref{eq: Def of Phi}), which we consider as an operator
\[
        \Phi: B(\bar{g},1/2) \subseteq C^{2,\alpha}\big(\Sym^2(T^\ast M)\big) \to
        C^{0,\alpha}\big(\Sym^2(T^\ast M)\big).
\]
Denote by $\mathcal{L}=(D\Phi)_{\bar{g}}$ the linearization of $\Phi$ at the background metric $\bar{g}$. By \Cref{prop: L is invertible} there exists a universal constant $C_0=C_0(n,\alpha,K,\Lambda)$ such that $\mathcal{L}$ is invertible with $||\mathcal{L}||_{\rm op}, ||\mathcal{L}^{-1}||_{\rm op} \leq C_0$. Moreover, by possibly enlarging $C_0$, the map $g \mapsto (D\Phi)_g$ is $C_0$-Lipschitz. Indeed, this follows from the same argument as in Step 1 in the proof of \Cref{prop: L is invertible} (also see \cite[Lemma 5.2]{HJ22}).

Therefore, applying (a quantitative version of) the inverse function theorem implies that there exist constants $\varepsilon_0^\prime=\varepsilon_0^\prime(n,\alpha,K,\Lambda) > 0$ and $C_0^\prime=C_0^\prime(n,\alpha,K,\Lambda)$ with the following property: For each $f \in C^{0,\alpha}\big(\Sym^2(T^\ast M)\big)$ with $||f-\Phi(\bar{g})||_0 \leq \varepsilon_0^\prime$ there exists a metric $g_f \in C^{2,\alpha}\big(\Sym^2(T^\ast M)\big)$ such that 
    \[
        \Phi(g_f)=f 
        \quad \text{and} \quad
        ||g_f-\bar{g}||_2 \leq C_0^\prime||f-\Phi(\bar{g})||_0.
    \]
Note that $\Phi(\bar{g})=\Ric(\bar{g})+(n-1)\bar{g}$. Since $||\cdot||_{C^{0,\alpha}} \leq C||\cdot||_{C^0}^{1-\alpha}||\cdot||_{C^1}^{\alpha}$, we have $||\Phi(\bar{g})||_0 \leq C \varepsilon^{1-\alpha}$ from the assumptions in \Cref{Main Theorem - introduction} and the \Cref{def: hybridnorm} of the hybrid norm $||\cdot||_0$. In particular, $f=0$ satisfies $||f-\Phi(\bar{g})|| \leq \varepsilon_0^\prime$ for $\varepsilon >0$ small enough. Thus, there exists a metric $g_0$ on $M$ such that
    \[
        \Phi(g_0)=0
        \quad \text{and} \quad
        ||g_0-\bar{g}||_2 \leq C \varepsilon^{1-\alpha}.
    \]
In particular, for $\varepsilon > 0$ small enough, $\sec(M,g_0) \leq -K/2 < 0$ as $\sec(M,\bar{g}) \leq - K < 0$ by \Cref{Main Theorem - introduction}(ii). Therefore, $\Phi(g_0)=0$ implies $\Ric(g_0)+(n-1)g_0=0$ due to \Cref{Zeros of Phi are Einstein}. This completes the proof.
\end{proof}

\Cref{cor: bounded negative curvature} now follows easily from \Cref{Main Theorem - introduction}.

\begin{proof}[Proof of \Cref{cor: bounded negative curvature}]
All constants implicitely depend on the choice of a Margulis constant \(\mu\). Therefore, if the constants we produce are allowed to depend on a constant \(\iota > 0\), we can without loss of generality assume \(\mu \leq \iota\). Thus, $\sec(M,\bar{g})=-1$ in $M_{\rm thin}$ since $\sec=-1$ outside $\Omega$ and ${\rm inj}(\Omega) \geq \iota$ by the assumptions of \Cref{cor: bounded negative curvature}. 

Moreover,
\[
	||{\rm Ric}(\bar{g})+(n-1)\bar{g}||_{L^2(M,\bar{g})}\leq {\rm vol}(\Omega)||{\rm Ric}(\bar{g})+(n-1)\bar{g}||_{C^0(M,\bar{g})} \leq v \varepsilon
\]
since $\sec=-1$, and hence ${\rm Ric}(\bar{g})+(n-1)\bar{g}=0$, outside $\Omega$. 

Therefore, for $\varepsilon > 0$ small enough (depending on $v$), the assumptions in \Cref{Main Theorem - introduction} follow from those in \Cref{cor: bounded negative curvature}, so that we can apply \Cref{Main Theorem - introduction} to obtain an Einstein metric on $M$ which is $C^{2,\alpha}$-close to $\bar{g}$.
\end{proof}

Finally, \Cref{cor: almost hyperbolic} trivially follows from \Cref{cor: bounded negative curvature}.

\bigskip

\noindent
\textsc{Département de Mathématiques, Université Libre de Bruxelles, Campus de la Plaine, CP 210, Boulevard du Triomphe, B-1050 Bruxelles, Belgique}\\
e-mail: frieder.jackel@ulb.be

\end{document}